\documentclass[a4paper,12pt]{article}

\usepackage{graphicx}
\usepackage{cmap}					% поиск в PDF
\usepackage[T2A]{fontenc}			% кодировка
\usepackage[utf8]{inputenc}			% кодировка исходного текста
\usepackage[english]{babel}	% локализация и переносы

\usepackage{amsmath,amssymb,amscd,amsthm}
\usepackage[all]{xy}

\newcommand{\lk}{\mathop{\rm lk}\nolimits}
\DeclareMathOperator{\im}{im}

\newtheorem{thm}{Theorem}%[section]

\newtheorem{prop}[thm]{Proposition}

\newtheorem{Ex}{Example}

\newtheorem{lemma}[thm]{Lemma}

\newtheorem{rem}[thm]{Remark}
\newtheorem{dfn}[thm]{Definition}

\author{Oleg R. Musin and Timur Shamazov}
\title{Computing the Hopf invariant}
\date{}

\begin{document}
\maketitle

\begin{abstract}
We consider the simplicial (cochain) counterpart of Whitehead's integral formula and
propose an algorithm for computing the Hopf invariant of a simplicial map
$f\colon S^{2n-1}\to S^n$ given by a labelling of the vertices of a triangulation.
In general the algorithm amounts to solving a sparse linear system and runs in
$O(N^3)$ time, where $N$ is the size of the triangulation. We show, in addition,
that if the triangulation of $S^{2n-1}$ becomes a shellable ball after the removal of
one facet --- as is the case, for instance, for boundary complexes of simplicial
polytopes --- then the system can be solved incrementally along the shelling,
and the Hopf invariant is computed with $O(N)$ arithmetic operations.
\end{abstract}

\section{Introduction}

In 1931 Heinz Hopf \cite{Hopf} defined a map $h:S^3 \to S^2$ and proved that $h$ is not
homotopic to the constant map. The {\em Hopf invariant} of a smooth or simplicial map
$f:S^3\to S^2$ is the linking number
$$
H(f):=\lk\bigl(f^{-1}(x),f^{-1}(y)\bigr) \in {\mathbb Z},
$$
where $x\ne y \in S^2$ are generic points, so that $f^{-1}(x)$ and $f^{-1}(y)$ are
disjoint (possibly empty) unions of circles; see \cite{Milnor}.

The map $H$ induces an isomorphism $\pi_3(S^2)\to{\mathbb Z}$, i.e. it completely
classifies the homotopy classes of maps from $S^3$ to $S^2$. Therefore, in order to
understand whether a map is null-homotopic, it is sufficient to check whether its Hopf
invariant is equal to zero.

For continuous maps $f: S^{2n-1} \to S^n$ the Hopf invariant is a homomorphism
$H:\pi_{2n-1}(S^n)\to{\mathbb Z}$. It vanishes identically for odd $n$; for even $n$ it
is nontrivial, and by Adams' theorem \cite{Adams} its image is the whole of
${\mathbb Z}$ only for $n=2,4,8$, and equals $2{\mathbb Z}$ otherwise. Steenrod
\cite{litlink3} defined the Hopf invariant using the cup product in the cohomology ring;
we recall this definition in Section 2. In the smooth case, there is {\em Whitehead's
integral formula} \cite{litlink2} that expresses the Hopf invariant through an integral
of a special differential form over $S^{2n-1}$. A generalization of this formula to
multivalued mappings was proposed by S.\,P.~Novikov \cite{Nov84}. Let $\omega$ be an
$n$-form on $S^n$ such
that $\int_{S^n}\omega = 1$. Then we may consider its pullback $f^{*}\omega$, which is an
$n$-form on $S^{2n - 1}$. Since $H^n_{dR}(S^{2n-1})=0$, there is an $(n-1)$-form $\theta$
on $S^{2n-1}$ such that $d\theta = f^{*}\omega$, and the Hopf invariant is given by
\begin{equation}
   H(f) = \int\limits_{S^{2n - 1}}\theta\wedge f^{*}\omega .
\end{equation}

\medskip

Let $f:S^m \to S^n$ be a simplicial map with respect to triangulations $T_1$ of $S^m$
and $T_2$ of $S^n$. Then $f$ is uniquely determined by the map of vertex sets
$$L:V(T_1)\to V(T_2),$$
which has to satisfy the following {\em simpliciality condition}: for every simplex
$\sigma\in T_1$ the set $L(V(\sigma))$ spans a simplex of $T_2$. Conversely, any such $L$
extends linearly (affinely on each simplex) to a simplicial map
$$f_L: {S}^{m} \rightarrow {S}^n,$$
and it is clear that $f_L$ coincides with $f$. Suppose that a colouring (labelling) $L$
is given. The question is how to find $[f_L]\in\pi_m(S^n)$? This problem has been solved
only in some special cases, see \cite{AM,Mad,MS,MS2,Mus14,Mus16,MusQS}.

Throughout the paper $T_2$ is the boundary triangulation of the $(n+1)$-simplex
$\Delta^{n+1}$; we denote it by $S_{n+2}^n$ and identify its vertex set with
$\{1,\dots,n+2\}$. Thus a simplicial map is given by a colouring
$$L:V(T_1) \to V(S_{n+2}^n)=\{1,\dots,n+2\},$$
and, since every proper subset of $\{1,\dots,n+2\}$ spans a face of
$\partial\Delta^{n+1}$, the simpliciality condition takes the following purely
combinatorial form:
\begin{equation}
L(V(\sigma))\ne\{1,\dots,n+2\}\quad\text{for every simplex }\sigma\in T_1,
\end{equation}
that is, {\em no simplex of $T_1$ carries the complete set of $n+2$ labels}. Fully
labelled $n$-simplices, i.e. $n$-simplices whose $n+1$ labels are pairwise distinct,
are precisely the simplices mapped nondegenerately onto facets of $S_{n+2}^n$. After a
target facet and its dual cochain $\omega$ are fixed, the support of $f^*\omega$ consists
of those $n$-simplices that are mapped onto that particular facet.

Motivation comes from obtaining estimates of the invariant $\mu(d)$, which appears in the
quantitative Sperner theorem from \cite{MusQS}. Let a triangulation of $D^4$ be labelled
by $\{1,2,3,4\}$ and suppose that no tetrahedron in its boundary is fully labelled. The
boundary labelling then defines a simplicial map $f:S^3\to S^2$, and the number of fully
labelled tetrahedra in $D^4$ is at least $\mu([f])$, where $[f]\in\pi_3(S^2)$ is the
homotopy class of the boundary map.

In this paper we consider the cochain analogue of (1) for simplicial maps
$f: S^{2n-1} \to S^n$ and propose an algorithm for computing the Hopf invariant
(Section 3). In Section 4 we show that for shellable triangulations the linear algebra
disappears completely: the cochain $\theta$ can be constructed facet by facet along a
shelling, which gives an algorithm linear in the number of facets.

\section{The Hopf invariant for simplicial maps}

Steenrod \cite{litlink3}, see also \cite{BT82,Sul,Wh57}, considered differential forms,
the Hopf invariant, and the integral Whitehead formula for triangulations, simplicial and
cell complexes. We follow the simplicial cochain version. From now on we fix a linear
order on the vertex set of every complex involved; this is needed to define the
Alexander--Whitney cup product.

Let $ f: S^{2n - 1} \to S^{n} $ be a simplicial map, $n\ge 2$.
To calculate the Hopf invariant, take an $ n $-cocycle $ \omega \in C^n(S^n;\mathbb{Z})$
representing a generator of $ H^{n}(S^{n};\mathbb{Z}) $, and consider its pullback
$ f^{*}\omega \in C^{n}(S^{2n-1};\mathbb{Z}) $.
Since $ H^{n}(S^{2n-1};\mathbb{Z}) = 0 $, there exists an $(n-1)$-cochain $ \theta $ with
integer values such that $ \delta \theta = f^{*}\omega $.
It is easy to verify that the cochain $ \theta \smile f^{*}\omega $ is a cocycle. Indeed,
by the Leibniz rule
\[
\delta(\theta \smile f^{*}\omega) = \delta\theta \smile f^{*}\omega \;+\;
(-1)^{n-1}\, \theta \smile \delta f^{*}\omega
= f^{*}\omega \smile f^{*}\omega \;+\; 0 = 0,
\]
because $ \delta f^{*}\omega = f^{*}(\delta\omega) = 0 $ and
$ f^{*}\omega \smile f^{*}\omega \in C^{2n}(S^{2n - 1}) $ vanishes for dimensional
reasons.

Here, the \textbf{coboundary operator} $ \delta: C^{k}(X) \to C^{k+1}(X) $ is defined on a
$ k $-cochain $ \alpha $ by
\[
(\delta\alpha)([v_0, \dots, v_{k+1}]) = \sum_{i=0}^{k+1} (-1)^{i}
\alpha([v_0, \dots, \hat{v}_i, \dots, v_{k+1}]),
\]
where $ [v_0, \dots, v_{k+1}] $ is an oriented $(k+1)$-simplex and $ \hat{v}_i $ denotes
omission of the vertex $ v_i $.
The \textbf{cup product} $ \smile: C^{k}(X) \times C^{\ell}(X) \to C^{k+\ell}(X) $ is
defined on simplices with increasingly ordered vertices by
\[
(\alpha \smile \beta)([v_0, \dots, v_{k+\ell}]) = \alpha([v_0, \dots, v_k]) \cdot
\beta([v_k, \dots, v_{k+\ell}]).
\]
This operation induces the cup product in cohomology, which is associative,
graded-commutative and natural with respect to simplicial maps.

Therefore we may consider the cohomology class of the cocycle
$\theta \smile f^{*}\omega$. It lies in $H^{2n-1}(S^{2n-1};\mathbb{Z})\cong\mathbb{Z}$,
hence it is an integer multiple of the chosen generator, and this integer is the Hopf
invariant $H(f)$. Of course, one has to check that nothing depends on the arbitrary
choices we made.

\begin{lemma}\label{lem:welldef}
The class $[\theta \smile f^{*}\omega]\in H^{2n-1}(S^{2n-1};\mathbb{Z})$ does not depend
on the choice of $\theta$ with $\delta\theta=f^*\omega$. The resulting class is also
independent of the cocycle representative $\omega$ of the chosen generator and of the
linear ordering of the vertices used to define the Alexander--Whitney product. Replacing
$\omega$ by $-\omega$ and $\theta$ by $-\theta$ leaves this class unchanged.
Consequently, the integer
    \begin{equation}
        H(f) = \bigl\langle \theta \smile f^{*}\omega, \; [T_1] \bigr\rangle,
    \end{equation}
where $[T_1]\in H_{2n-1}(S^{2n-1};\mathbb{Z})$ is the fundamental class of the
triangulation $T_1$ of $S^{2n-1}$, is independent of the orientation of the target sphere,
changes sign when the orientation of the source sphere is reversed, and depends only on
the homotopy class of $f$.
\end{lemma}

\begin{proof}
Let $\theta_1,\theta_2$ be two solutions and $\eta=\theta_1-\theta_2$. Then
$\delta\eta=0$, and since $H^{n-1}(S^{2n-1};\mathbb{Z})=0$ for $n\ge 2$ (here we use
$1\le n-1\le 2n-2$), there is $\xi\in C^{n-2}$ with $\eta=\delta\xi$. Hence
\[
\eta \smile f^{*}\omega=\delta\xi\smile f^{*}\omega
=\delta\bigl(\xi\smile f^{*}\omega\bigr),
\]
again by the Leibniz rule and $\delta f^*\omega=0$; so the two cocycles differ by a
coboundary. If $\omega$ is replaced by $-\omega$, then $-\theta$ is a corresponding
primitive and $(-\theta)\smile f^*(-\omega)=\theta\smile f^*\omega$.

The remaining assertions are the standard invariance properties of Steenrod's functional
cup product. In the present case its indeterminacy vanishes because the intermediate
cohomology groups of $S^{2n-1}$ vanish, and the resulting class is the classical Hopf
invariant; see \cite{litlink3}.
\end{proof}

\begin{lemma}\label{lem:system}
Let $ T_1 $ be a finite simplicial triangulation of $ S^{2n-1} $. A cochain
$ \theta \in C^{n-1}(T_1; \mathbb{Q}) $ satisfying $ \delta\theta = f^{*}\omega $ can be
obtained by solving a linear system over $ \mathbb{Q} $ with the following structure:
\begin{enumerate}
    \item The variables are the values of $ \theta $ on the $ (n-1) $-simplices of
    $ T_1 $; their number is $f_{n-1}$.
    \item For each $ n $-simplex $ \sigma $ of $ T_1 $ the equation
    \begin{equation}
        (\delta\theta)(\sigma) = f^{*}\omega(\sigma)
    \end{equation}
    provides one linear equation; their number is $f_n$.
    \item In each such equation exactly $ n+1 $ variables appear (corresponding to the
    $ (n-1) $-faces of $ \sigma $), and their coefficients are $ \pm 1 $.
\end{enumerate}
Because $ H^{n}(S^{2n-1}; \mathbb{Q}) = 0 $, this system is consistent; moreover, since
$H^{n}(S^{2n-1}; \mathbb{Z}) = 0$ and all cochain groups are free, it admits an integral
solution.
\end{lemma}

In the case $n = 2$ each equation has the form
$$\langle \theta, v_0v_1\rangle + \langle \theta, v_1v_2\rangle +
\langle \theta , v_2v_0\rangle = \langle f^{*}\omega, v_0v_1v_2\rangle$$
for every $2$-simplex $v_0v_1v_2$ in $S^3$. That is, the simplicial analogue of Stokes'
formula must be satisfied.
\begin{center}
\includegraphics[width=0.4\textwidth]{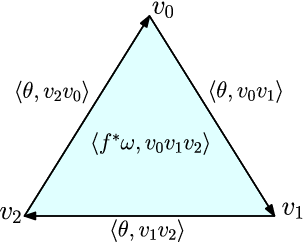}
\end{center}

A solution of this system exists, but it is far from unique: the solution set is an affine
subspace parallel to $Z^{n-1}=\ker\bigl(\delta:C^{n-1}\to C^{n}\bigr)$.

\begin{thm}\label{thm:dim}
The dimension of the (affine) solution space equals
$$ f_{n-2} - f_{n - 3} + f_{n-4} - \ldots + (-1)^{n-2}f_0 + (-1)^{n-1}
=\sum_{i=0}^{n-2}(-1)^{n-2-i}f_i+(-1)^{n-1},$$
where $f_i$ is the number of $i$-simplices of the triangulation $T_1$ of $S^{2n-1}$.
\end{thm}

\begin{proof}
The solution space is a coset of $Z^{n-1}:=\ker(\delta:C^{n-1}\to C^n)$, so we have to
compute $\dim Z^{n-1}$. We claim that
\[
\dim Z^{k}=f_{k-1}-f_{k-2}+\dots+(-1)^{k-1}f_0+(-1)^{k}\qquad (0\le k\le n-1),
\]
and prove this by induction on $k$. For $k=0$ we have $Z^0=H^0(S^{2n-1};\mathbb{Q})
=\mathbb{Q}$, i.e. $\dim Z^0=1$, as claimed. Let $1\le k\le n-1$. Since
$0<k<2n-1$, we have $H^{k}(S^{2n-1};\mathbb{Q})=0$, hence
\[
\dim Z^{k}=\dim B^{k}=\dim\im\bigl(\delta:C^{k-1}\to C^{k}\bigr)
=\dim C^{k-1}-\dim Z^{k-1}=f_{k-1}-\dim Z^{k-1},
\]
by the rank--nullity theorem, and it remains to apply the induction hypothesis.
Putting $k=n-1$ gives the statement.
\end{proof}

\begin{rem}
For $n=2$ the theorem gives $f_0-1$, the number of edges of a spanning tree of the
$1$-skeleton --- see Proposition \ref{prop:gauge} below.
\end{rem}

As a consequence, for the computation of the Hopf invariant one may {\em normalize}
$\theta$: choose a set $S$ of $(n-1)$-simplices of $T_1$ such that the restriction map
$Z^{n-1}\to\mathbb{Q}^{S}$, $\eta\mapsto\eta|_S$, is an isomorphism (in particular
$|S|=\dim Z^{n-1}$), prescribe the values of $\theta$ on $S$ arbitrarily, and then the
solution of
\[
\delta \theta = f^{*}\omega ,\qquad \theta|_S \text{ prescribed},
\]
is unique. Here $f^*\omega\in C^n(T_1;\mathbb{Z})$ is the pullback of a cocycle
representing the generator of $H^n(S^n;\mathbb{Z})$; note that the {\em class} of
$f^*\omega$ in $H^n(S^{2n-1};\mathbb{Z})$ is of course zero --- this is exactly why
$\theta$ exists.

\begin{prop}\label{prop:gauge}
Let $n=2$, so that $f: S^3\to S^2$. Let $G$ be a spanning tree of the $1$-skeleton of
$T_1$ and let $S=E(G)$. Then the restriction $Z^{1}\to \mathbb{Q}^{S}$ is an isomorphism.
Hence there is a unique cochain $\theta$ with $\delta\theta=f^*\omega$ and
$\theta(\tau)=0$ for all edges $\tau\in G$.
\end{prop}

\begin{proof}
Since $H^1(S^3;\mathbb{Q})=0$, we have $Z^1=B^1=\delta C^0$. If $\delta g|_{E(G)}=0$ for
$g\in C^0$, then $g$ is constant along every edge of $G$, hence constant on $S^3$ (as $G$
is spanning and connected), hence $\delta g=0$. So the restriction map is injective; it is
surjective because $\dim Z^1=f_0-1=|E(G)|$ by Theorem \ref{thm:dim}.
\end{proof}

\section{The algorithm and its computational complexity}

Let $S^n_{n+2}=\partial\Delta^{n+1}$ be the boundary triangulation of the $(n+1)$-simplex,
with vertices labelled by $1<2<\dots<n+2$, and let $ T_1 $ be a finite simplicial
triangulation of the sphere $ S^{2n-1} $ with a fixed linear order on its vertices. Let a
simplicial map $ f: T_1 \to S^n_{n+2} $ be defined by a vertex labelling
$ \ell(v) \in \{1,2,\dots, n + 2\} $ satisfying (2). The Hopf invariant $ H(f) $ can be
computed as follows.

\begin{enumerate}
    \item Fix the $ n $-simplex $ \bar{\sigma} = [1,2, \dots, n+1] $ of $S^n_{n+2}$
    (i.e. the facet opposite to the vertex $n+2$) and let $\omega\in C^n(S^n_{n+2})$ be
    the cochain dual to it:
    \[
    \omega(\sigma) =
    \begin{cases}
        1, & \sigma = \bar{\sigma}, \\
        0, & \text{for the other $n$-simplices of $S^n_{n+2}$.}
    \end{cases}
    \]
    Since $\langle\omega,[S^n_{n+2}]\rangle=\pm1$, the class $[\omega]$ generates
    $H^n(S^n;\mathbb{Z})$.

    \item For each $ n $-simplex $ \sigma = [v_0<\dots< v_n] $ of $ T_1 $ set
    \[
    f^{*}\omega(\sigma)=
    \begin{cases}
    \operatorname{sgn}(\pi), & \text{if }\{\ell(v_0),\dots,\ell(v_n)\}=\{1,\dots,n+1\},\\
    0,&\text{otherwise},
    \end{cases}
    \]
    where in the first case $\pi$ is the permutation that rearranges
    $(\ell(v_0),\dots,\ell(v_n))$ in increasing order. Thus
    $ f^{*}\omega(\sigma) \in \{-1, 0, +1\} $, and it is nonzero exactly on those fully
    labelled $n$-simplices whose label set is $\{1,\dots,n+1\}$.

    \item Introduce one variable $ x_{\tau} $ for each $ (n-1) $-simplex
    $\tau=[v_0<\dots<v_{n-1}]$ of $ T_1 $; it represents the value $ \theta(\tau) $. For
    each $ n $-simplex $ \sigma = [v_0<\dots< v_n] $ we get the equation
    \[
    \sum_{i=0}^{n} (-1)^i \, x_{[v_0, \dots, \widehat{v_i}, \dots, v_n]}
    = f^{*}\omega(\sigma),
    \]
    where $ \widehat{v_i} $ denotes omission. This is a linear system with
    $ f_n=|T_1^{(n)}| $ equations and $ f_{n-1}=|T_1^{(n-1)}| $ unknowns; compute any
    solution over $\mathbb{Q}$ (say, by Gaussian elimination); by Lemma
    \ref{lem:system} it is consistent.

    \item Fix the fundamental cycle $ [T_1] = \sum_{\sigma \in T_1^{(2n-1)}}
    \varepsilon_{\sigma} \, \sigma $ (where $ \varepsilon_{\sigma} = \pm 1 $ compares the
    chosen orientation of $S^{2n-1}$ with the increasing order of the vertices of
    $\sigma$). For each $ (2n-1) $-simplex $ \sigma = [u_0 <\dots< u_{2n-1}] $ compute
    \[
    (\theta \smile f^{*}\omega)(\sigma) = \theta\bigl([u_0, \dots, u_{n-1}]\bigr)
     \cdot  f^{*}\omega\bigl([u_{n-1}, \dots, u_{2n-1}]\bigr),
    \]
    and output
    \[
    H(f) = \bigl\langle \theta \smile f^{*}\omega, [T_1] \bigr\rangle
          = \sum_{\sigma \in T_1^{(2n-1)}} \varepsilon_{\sigma} \;
          (\theta \smile f^{*}\omega)(\sigma).
    \]
\end{enumerate}

\begin{thm}
The above algorithm computes the Hopf invariant of the simplicial map
$ f: S^{2n-1} \to S^{n} $ defined by a vertex labelling
$ \ell: T_1^{(0)} \to (S^n_{n+2})^{(0)} $. The result is an integer $ H(f) $ which does
not depend on the choice of the solution $\theta$ and depends only on the homotopy class
of $ f $.
\end{thm}

\begin{proof}
Steps 1--2 produce a cocycle representing a generator of $H^n(S^n;\mathbb{Z})$ and its
pullback; step 3 solves $\delta\theta=f^*\omega$; step 4 evaluates the cocycle
$\theta\smile f^*\omega$ on the fundamental cycle. The correctness and the independence of
the choices follow from Lemma \ref{lem:welldef}. Note that although $\theta$ is computed
over $\mathbb{Q}$, the resulting pairing is an integer, since some integral solution
$\theta_{\mathbb Z}$ exists and the answer does not depend on the solution.
\end{proof}

\begin{thm}
Let $N$ be the total number of simplices of $T_1$. The complexity of the algorithm is
$O(N^3)$; more precisely, it is $O(f_{n-1}^2 f_n + f_{2n-1})$ arithmetic operations.
\end{thm}

\begin{proof}
The most expensive step is step 3, i.e. solving a linear system with $f_n$ equations and
$f_{n-1}$ unknowns, which by Gaussian elimination takes $O(f_{n-1}^2f_n)$ operations.
Steps 1, 2 and 4 are linear in the number of simplices of the corresponding dimensions,
the last one contributing $O(f_{2n-1})$.
\end{proof}

We observe that the matrix of our system is sparse: each row contains only $n + 1$
nonzero entries, all equal to $\pm 1$. Such systems can be solved by various iterative and
sparse direct methods; we refer to \cite{Saad} and \cite{Davis}. In the special case when
$n$ is odd every map $S^{2n-1}\to S^n$ has Hopf invariant zero, so the computation
reduces to $O(1)$.
\section{Shellable triangulations: a linear-time algorithm}

A shelling of $T$ allows one to build $\theta$ one facet at a time. At each nontrivial
extension step we choose a vertex $v_0$, the apex, which lies in every boundary facet of
the restriction complex $R_k$; the new values of $\theta$ at the faces containing $v_0$
may be set to zero, every remaining new value is then read off from a single equation,
and the other equations of the facet hold automatically.

Throughout this section $T$ is a triangulation of $S^{2n-1}$, $n\ge2$; its
$(2n-1)$-simplices are called {\em facets} and their number is $N$ (for fixed $n$ the
total number of simplices of $T$ is $\Theta(N)$). For a simplex $\sigma$ with vertex set
$V$ we write $\bar\sigma$ for the complex of its faces and $\partial_v$ for the facet of
$\partial\sigma$ opposite to $v\in V$. If $\tau=[u_0,\dots,u_{k}]$ is a face with
$v_0\notin\tau$, we write $v_0\tau:=[v_0,u_0,\dots,u_{k}]$ and
$\tau_i:=[u_0,\dots,\widehat{u_i},\dots,u_k]$. Here and below we use oriented
(alternating) cochains, so that permuting the vertices multiplies a value by the sign of
the permutation; the global vertex order fixed in Section 2 is needed only for the cup
product in the last step of the algorithm.

\begin{dfn}
An ordering $\sigma_1,\dots,\sigma_N$ of the facets of $T$ is a {\em shelling} if for
every $k\ge2$ the {\em restriction complex}
$R_k:=\bar\sigma_k\cap\bigl(\bigcup_{j<k}\bar\sigma_j\bigr)$
is a nonempty union of facets of $\partial\sigma_k$. We put $R_1:=\emptyset$ and
$B_k:=\bigcup_{j\le k}\bar\sigma_j$.
\end{dfn}

Equivalently, $T$ is shellable if and only if it becomes a shellable ball after the
removal of one open facet: one only has to put the deleted facet last. By the
Bruggesser--Mani theorem \cite{BM71}, see also \cite[Ch.~8]{Zie95}, boundary complexes of
simplicial polytopes are shellable, and a line shelling is easy to produce. Non-shellable
triangulations of spheres exist \cite{Lic91} and deciding shellability is NP-hard
\cite{GPPTW}, so we assume a shelling to be part of the input.

At every step for which $R_k\ne\partial\sigma_k$, shellability will be used only through
the following property: $R_k$ is a union of facets of $\partial\sigma_k$ omitting at least
one of them, so that an {\em apex} belonging to all facets in this union is available.

\begin{lemma}[Apex rule]\label{lem:apex}
Let $\sigma$ be a $(2n-1)$-simplex with vertex set $V$, $n\ge2$, let
$R=\bigcup_{v\in W}\overline{\partial_v}$ with $W\subsetneq V$ (the case $W=\emptyset$,
$R=\emptyset$, is allowed), and fix an apex $v_0\in V\setminus W$. Let
$g\in C^{n}(\bar\sigma;\mathbb{Z})$ be a cocycle and let
$\theta_R\in C^{n-1}(R;\mathbb{Z})$ satisfy $\delta\theta_R=g|_R$. Then there is exactly
one cochain $\theta\in C^{n-1}(\bar\sigma;\mathbb{Z})$ with
\[
{\rm(a)}\ \ \delta\theta=g,\qquad
{\rm(b)}\ \ \theta|_{R}=\theta_R,\qquad
{\rm(c)}\ \ \theta(\tau)=0 \ \text{ for every } \tau\notin R \text{ with } v_0\in\tau .
\]
It is computed as follows: the faces $\tau\notin R$ with $v_0\in\tau$ get the value $0$ by
{\rm(c)}, and the remaining faces, i.e. those $\tau\notin R$ with $v_0\notin\tau$, get
\begin{equation}\label{eq:apex}
\theta(\tau)=g(v_0\tau)+\sum_{i=0}^{n-1}(-1)^{i}\,\theta(v_0\tau_i),
\end{equation}
where every $v_0\tau_i$ contains $v_0$ and hence is already known from {\rm(b)} or
{\rm(c)}.
\end{lemma}

\begin{proof}
{\em Uniqueness and formula \eqref{eq:apex}.} Let $\theta$ satisfy (a)--(c) and let
$\tau\notin R$ with $v_0\notin\tau$. All faces of the $n$-simplex $v_0\tau$ other than
$\tau$ itself are of the form $v_0\tau_i$ and contain $v_0$, so their values are
prescribed by (b) or (c). Writing out $(\delta\theta)(v_0\tau)=g(v_0\tau)$ and solving for
$\theta(\tau)$ gives exactly \eqref{eq:apex}. Thus (a)--(c) leave no freedom.

{\em Existence.} For $k\ge1$ define $H:C^{k}(\bar\sigma)\to C^{k-1}(\bar\sigma)$ by
$(H\alpha)(\tau):=\alpha(v_0\tau)$ if $v_0\notin\tau$ and $(H\alpha)(\tau):=0$ if
$v_0\in\tau$. Then $\delta H+H\delta=\mathrm{id}$. Indeed, if $v_0\in\tau$, say
$\tau=v_0\rho$, then $H\delta\alpha(\tau)=0$, while in $\delta H\alpha(\tau)$ only the
term corresponding to the omission of $v_0$ survives and equals
$(H\alpha)(\rho)=\alpha(\tau)$; and if $v_0\notin\tau$, then
\[
(H\delta\alpha)(\tau)=(\delta\alpha)(v_0\tau)
=\alpha(\tau)-\sum_i(-1)^{i}\alpha(v_0\tau_i),\qquad
(\delta H\alpha)(\tau)=\sum_i(-1)^{i}\alpha(v_0\tau_i).
\]
Now let $\theta'$ be the extension of $\theta_R$ by zero and $\mu:=g-\delta\theta'$. Then
$\delta\mu=0$ and $\mu|_R=g|_R-\delta\theta_R=0$. Put $\theta:=\theta'+H\mu$. Property (a)
holds because $\delta H\mu=\mu-H\delta\mu=\mu$; property (c) holds because both summands
vanish on such $\tau$. For (b) take $\tau\in R$: if $v_0\in\tau$ then $(H\mu)(\tau)=0$; if
$v_0\notin\tau$ then $\tau\subseteq\partial_v$ for some $v\in W$, and $v_0\in\partial_v$
since $v_0\ne v$, whence $v_0\tau\subseteq\partial_v\subset R$ and
$(H\mu)(\tau)=\mu(v_0\tau)=0$. (This is the only place where the existence of the apex is
used: it makes $R$ a cone with apex $v_0$.) Finally, $\theta$ is integral because $H$ and
$\delta$ preserve integrality.
\end{proof}

\begin{Ex}
Let $n=2$, $\sigma=[a,b,c,d]$, $R=\overline{[b,c,d]}$, so $W=\{a\}$ and we may take
$v_0=b$. The new edges are $ab,ac,ad$; rule (c) gives $\theta(ab)=0$, and \eqref{eq:apex}
determines $\theta(ac)$ from the triangle $abc$ and $\theta(ad)$ from $abd$, each time
from a single equation whose two other edges are $ab$ and an edge of $R$. The equation of
the fourth triangle $acd$ then holds automatically.
\end{Ex}

\begin{thm}\label{thm:shell}
Let $T$ be a triangulation of $S^{2n-1}$, $n\ge2$, given together with a shelling
$\sigma_1,\dots,\sigma_N$, and let $f$ be a simplicial map defined by a labelling $\ell$
as in Section 3. For $k=1,\dots,N$, do nothing if $R_k=\partial\sigma_k$; otherwise apply
Lemma \ref{lem:apex} to $\sigma=\sigma_k$, $R=R_k$,
$g=f^{*}\omega|_{\bar\sigma_k}$ and $\theta_R=\theta|_{R_k}$, with an arbitrary apex.
Then every step is well defined and leaves the previously computed values untouched, the
resulting cochain $\theta\in C^{n-1}(T;\mathbb{Z})$ is integral and
satisfies $\delta\theta=f^{*}\omega$ on the whole of $T$, and
\[
H(f)=\bigl\langle \theta\smile f^{*}\omega,[T]\bigr\rangle .
\]
For fixed $n$ the algorithm performs $O(N)$ arithmetic
operations, at most $n\binom{2n}{n}$ additions per facet.
\end{thm}

\begin{proof}
We show by induction on $k$ that $\theta$ is defined on all $(n-1)$-simplices of $B_k$ and
satisfies $\delta\theta=f^{*}\omega$ on all $n$-simplices of $B_k$; for $k=N$ this is the
assertion, since $B_N=T$.

Let the claim hold for $k-1$. The cochain $g=f^{*}\omega|_{\bar\sigma_k}$ is a cocycle,
since $\delta f^{*}\omega=f^{*}\delta\omega=0$, and $R_k\subseteq B_{k-1}$, so
$\theta|_{R_k}$ is already defined and satisfies $\delta(\theta|_{R_k})=g|_{R_k}$.

If $R_k=\partial\sigma_k$, there is nothing to do: as $n\le 2n-2$, every $(n-1)$- and
every $n$-face of $\sigma_k$ is a proper face, hence lies in $R_k\subseteq B_{k-1}$, and
the claim for $k$ coincides with the claim for $k-1$. Otherwise $R_k$ is a union of
facets of $\partial\sigma_k$ omitting at least one of them (for $k=1$ it is empty), so an
apex exists and Lemma \ref{lem:apex} applies. It assigns values only to the
$(n-1)$-faces of $\sigma_k$ outside $R_k=\bar\sigma_k\cap B_{k-1}$, i.e. to faces not yet
treated, and yields $\delta\theta=f^{*}\omega$ on all $n$-simplices of $\bar\sigma_k$.
Together with the induction hypothesis this proves the claim for $k$; the formula for
$H(f)$ is then Lemma \ref{lem:welldef}.

For the complexity, a facet has $\binom{2n}{n}$ faces of dimension $n-1$, and each
application of \eqref{eq:apex} costs $n$ additions; evaluating $\theta\smile f^{*}\omega$
costs $O(1)$ per facet.
\end{proof}

\bigskip

\noindent {\bf {Acknowledgments.}}
We thank Dennis Sullivan for suggesting the idea of computing the Hopf invariant via
Whitehead's formula for simplicial maps, and Sergei Fomin for helpful comments and
discussions.

\medskip

Timur Shamazov is supported by the ``Priority 2030'' strategic academic leadership
program. The author thanks the Summer Research Programme at MIPT --- LIPS-25 for the
opportunity to work on this and other problems.

\bigskip

\medskip

\noindent O. R. Musin, University of Texas Rio Grande Valley, School of Mathematical and Statistical Sciences, One West University Boulevard, Brownsville, TX, 78520, USA.

\noindent {\it E-mail address:} oleg.musin@utrgv.edu

\medskip

\noindent T. Shamazov, Saint Petersburg State University, Department of Mathematics and Computer Science

\noindent {\it E-mail address:} shamazovt@inbox.ru

\end{document}